\documentclass[a4paper,10pt,reqno]{article}

\usepackage{amssymb,latexsym}
\usepackage{mltex}
\usepackage{amsmath}
 \usepackage{lscape}
\usepackage{enumerate}
\usepackage{amsthm}
\usepackage{hyperref}

\newtheorem{thm}{Theorem}

\newtheorem{lem}{Lemma}[section]
\newtheorem{prop}[lem]{Proposition}
\newtheorem{rem}{Remark}

\newtheorem*{thrm}{Theorem}

\newcommand{\nablab}{\overline{\nabla}}
\newcommand{\nablat}{\widetilde{\nabla}}

\newcommand{\lgra}{\longrightarrow}

\newcommand{\iid}{\mathrm{Id}\,}

\newcommand{\eend}{\mathrm{End}\,}

\newcommand{\trace}{\mathrm{tr\,}}

\newcommand{\lto}{\ensuremath{\longrightarrow}}
\newcommand{\C}{\mathbb{C}}

\newcommand{\HH}{\mathbb{H}}

\newcommand{\R}{\mathbb{R}}

\newcommand{\Mc}{\mathbb{M}^{4}(c)}
\newcommand{\A}{\mathcal{A}}

\newcommand{\pre}{\Re e}

\newcommand{\function}[5]
{\begin{eqnarray*}\begin{array}{r@{}ccl}
 #1\;\colon\;  & #2 &\lto & #3 \\[.05cm]
  & #4 &\longmapsto  & #5
\end{array}\end{eqnarray*}
}
\newcommand{\beqt}{\begin{equation}}  \newcommand{\eeqt}{\end{equation}}
\newcommand{\bal}{\begin{align}}      \newcommand{\eal}{\end{align}}
\newcommand{\ba}{\begin{array}}      \newcommand{\ea}{\end{array}}
\newcommand{\bc}{\begin{center}}     \newcommand{\ec}{\end{center}}
\newcommand{\be}{\begin{enumerate}}  \newcommand{\ee}{\end{enumerate}}
\newcommand{\beq}{\begin{eqnarray}}  \newcommand{\eeq}{\end{eqnarray}}
\newcommand{\beQ}{\begin{eqnarray*}} \newcommand{\eeQ}{\end{eqnarray*}}
\newcommand{\bi}{\begin{itemize}}    \newcommand{\ei}{\end{itemize}}
\newcommand{\bt}{\begin{tabular}}    \newcommand{\et}{\end{tabular}}
\newcommand{\finpreuve}{\hfill\square\\}

\newcommand{\thmrm}[1]{\text{\emph{#1}}}

\title{Spinorial Representation of Surfaces into 4-dimensional Space Forms}
\author{Pierre Bayard\footnote{bayard@ifm.umich.mx},  Marie-Am\'elie Lawn\footnote{mlawn@math.utexas.edu}  and  Julien Roth\footnote{julien.roth@univ-mlv.fr}}
\date{}
\begin{document}
\maketitle

\begin{abstract}
In this paper we give a geometrically invariant spinorial representation of surfaces in four-dimensional space forms. In the Euclidean space, we obtain a representation formula which generalizes the Weierstrass representation formula of minimal surfaces. We also obtain as particular cases the spinorial characterizations of surfaces in $\R^3$ and in $S^3$ given by T. Friedrich and by B. Morel. 
\end{abstract}
{\it keywords:} Dirac Operator, Isometric Immersions, Weierstrass Representation.

\date{}
\maketitle\pagenumbering{arabic}
\section{Introduction}
The Weierstrass representation describes a conformal minimal immersion of a Riemann surface $M$ into the three-dimensional Euclidean space $\R^3$. Precisely, the immersion is expressed using two holomorphic functions $f,g:M\longrightarrow \C$ by the following integral formula
$$(x_1,x_2,x_3)=\pre\left(\int f(1-g^2)dz,\int if(1+g^2)dz,\int 2fgdz\right):M\lgra\R^3.$$
On the other hand, the spinor bundle $\Sigma M$ over $M$ is a two-dimensional complex vector bundle splitting into
$$\Sigma M=\Sigma^+M\oplus\Sigma^-M=\Lambda^{0}M\oplus\Lambda^{0,1}M.$$
Hence, a pair of holomorphic functions $(g,f)$ can be considered as a spinor field $\varphi=(g,fdz)$. Moreover, the Cauchy-Riemann equations satisfied by $f$ and $g$ are equivalent to the Dirac equation
$$D\varphi=0.$$
This representation is still valid for arbitrary surfaces. In the general case, the functions $f$ and $g$ are not holomorphic and the Dirac equation becomes 
$$D\varphi=H\varphi,$$
where $H$ is the mean curvature of the immersion. This fact is well-known and has been studied in the last years by many authors (see \cite{Ko,KS,Ta,Ta2}). \\
\indent In \cite{Fr}, T. Friedrich gave a geometrically invariant spinorial representation of surfaces in $\R^3$. This approach was generalized to surfaces of other three-dimensional spaces  \cite{Mo,Ro} and also in the pseudo-Riemannian case \cite{La,LR}.\\
\indent
The aim of the present paper is to extend this approach to the case of codimension $2$ and then provide a geometrically invariant representation of surfaces in the $4$-dimensional space form $\Mc$ of sectional curvature $c$  by spinors solutions of a Dirac equation. 
\section{Preliminaries}
\subsection{The fundamental theorem of surfaces in $\Mc$}
Let $(M^2,g)$ be an oriented surface isometrically immersed into the four-dimensional space form $\Mc$. Let us denote by $E$ its normal bundle and by $B:TM\times TM\lgra E$ its second fundamental form defined by
$$B(X,Y)=\overline{\nabla}_XY-\nabla_XY,$$
where $\nabla$ and $\nablab$ are the Levi-Civita connections of $M$ and $\Mc$ respectively. For $\xi\in \Gamma(E)$, the shape operator associated to $\xi$ is defined by
$$S_{\xi}(X)=-\left(\overline{\nabla}_{X}\xi\right)^T,$$
for all $X\in\Gamma(TM),$ where the upper index $T$ means that we take the component of the vector tangent to $M$. Then, the following equations hold:
\be
\item $K=\left<B(e_1,e_1),B(e_2,e_2)\right>-|B(e_1,e_2)|^2+c, $  (Gauss equation) 
\item $K_N=-\left<(S_{e_3}\circ S_{e_4}-S_{e_4}\circ S_{e_3})(e_1),e_2\right> ,$ (Ricci equation) 
\item  $(\nabla^N_XB)(Y,Z)-(\nabla^N_YB)(X,Z)=0,$ (Codazzi equation)
\ee
where $K$ and $K_N$ are the curvatures of $(M,g)$ and $E,$ $(e_1,e_2)$ and $(e_3,e_4)$ are orthonormal and positively oriented bases of $TM$ and $E$ respectively, and where $\nabla^N$ is the natural connection induced on the normal bundle $T^*M^{\otimes 2}\otimes E$. Reciprocally, there is the following theorem:
\begin{thrm}[Tenenblat \cite{Te}]
Let $(M^2,g)$ be a Riemannian surface and $E$ a vector bundle of rank $2$ on $M,$ equipped with a metric $\langle.,.\rangle$ and a compatible connection. We suppose that $M$ and $E$ are oriented. Let $B:TM\times TM\lgra E$ be a bilinear map satisfying the Gauss, Ricci and Codazzi equations above, where, if $\xi\in E,$ the shape operator $S_\xi:TM\rightarrow TM$ is the symmetric operator such that
$$g\left( S_{\xi}(X),Y\right)=\langle B(X,Y),\xi\rangle$$
for all $X,Y\in TM.$ Then, there exists a local isometric immersion $V\subset M\lgra \Mc$ so that $E$ is identified with the normal bundle of $M$ into $\Mc$ and with $B$ as second fundamental form. 
\end{thrm}
\subsection{Twisted spinor bundle}
Let $(M^2,g)$ be an oriented Riemannian surface, with a given spin structure, and $E$ an oriented and spin vector bundle of rank 2 on $M$. We consider the spinor bundle $\Sigma$ over $M$ twisted by $E$ and defined by
$$\Sigma=\Sigma M\otimes\Sigma E.$$
We endow $\Sigma$ with the spinorial connection $\nabla$ defined by
$$\nabla=\nabla^{\Sigma M}\otimes\iid_{\Sigma E}+\iid_{\Sigma M}\otimes\nabla^{\Sigma E}.$$
We also define the Clifford product $\cdot$ by
$$\left\{\begin{array}{l} X\cdot\varphi=(X\cdot_{_M}\alpha)\otimes\overline\sigma\quad\text{if}\ X\in\Gamma(TM)\\ \\
X\cdot\varphi=\alpha\otimes(X\cdot_{_E}\sigma)\quad\text{if}\ X\in\Gamma(E)
\end{array}
\right.$$
for all $\varphi=\alpha\otimes\sigma\in\Sigma M\otimes\Sigma E,$ where $\cdot_{_M}$ and $\cdot_{_E}$ denote the Clifford products on $\Sigma M$ and on $\Sigma E$ respectively and where $\overline{\sigma}=\sigma^+-\sigma^-.$ We finally define the Dirac operator $D$ on $\Gamma(\Sigma)$ by
$$D\varphi=e_1\cdot\nabla_{e_1}\varphi+e_2\cdot\nabla_{e_2}\varphi,$$
where $(e_1,e_2)$ is an orthonormal basis of $TM.$
\\

We note that $\Sigma$ is also naturally equipped  with a hermitian scalar product $\langle.,.\rangle$ which is compatible to the connection $\nabla$, since so are $\Sigma M$ and $\Sigma E$, and thus also with a compatible real scalar product $\Re e\langle.,.\rangle.$ We also note that the Clifford product $\cdot$ of vectors belonging to $TM\oplus E$ is antihermitian with respect to this hermitian product. Finally, we stress that the four subbundles $\Sigma^{\pm\pm}:=\Sigma^{\pm}M\otimes\Sigma^{\pm}E$ are orthogonal with respect to the hermitian product. Throughout the paper we will assume that the hermitian product is $\C-$linear w.r.t. the first entry, and $\C-$antilinear w.r.t. the second entry.
\subsection{Spin geometry of surfaces in $\Mc$}
It is a well-known fact (see \cite{Ba,HZ}) that there is an identification between the spinor bundle $\Sigma\Mc_{|M}$ of $\Mc$ over $M,$  and the spinor bundle of $M$ twisted by the normal bundle $\Sigma:=\Sigma M\otimes\Sigma E$. Moreover, we have the spinorial Gauss formula: for any $\varphi\in\Gamma(\Sigma)$ and any $X\in TM$,
$$\nablat_X\varphi=\nabla_X\varphi+\frac{1}{2}\sum_{j=1,2}e_j\cdot B(X,e_j) \cdot\varphi,$$
where $\nablat$ is the spinorial connection of $\Sigma\Mc$ and $\nabla$ is the spinoral connection of $\Sigma$ defined by
$$\nabla\ =\ \nabla^{\Sigma M}\otimes\iid_{\Sigma E}\ +\ \iid_{\Sigma M}\otimes\nabla^{\Sigma E}.$$
Here $\cdot$ is the Clifford product on $\Mc.$ Therefore, if $\varphi$ is a Killing spinor of $\Mc$, that is satisfying
$$\widetilde{\nabla}_X\varphi=\lambda X\cdot \varphi, $$
where the Killing constant $\lambda$ is $0$ for the Euclidean space, $\pm\frac{1}{2}$ for the sphere and $\pm\frac{i}{2}$ for the hyperbolic space, that is, $4\lambda^2=c$, then its restriction over $M$ satisfies
\beqt\label{eqkilling0}
\nabla_X\varphi=-\frac{1}{2}\sum_{j=1,2}e_j\cdot B(X,e_j) \cdot\varphi+\lambda X\cdot \varphi.
\eeqt
Taking the trace in \eqref{eqkilling0}, we obtain the following Dirac equation
\beqt\label{eqdirac}
D\varphi=\vec{H}\cdot\varphi-2\lambda\varphi,
\eeqt
where we have again $D\varphi=\sum_{j=1}^2e_j\cdot \nabla_{e_j}\varphi$ and $\vec H=\frac{1}{2}\sum_{j=1}^2B(e_j,e_j)$ is the mean curvature vector of $M$ in $\Mc.$
\\

Let us consider $\omega_4=-e_1\cdot e_2\cdot e_3\cdot e_4$. We recall that $\omega_4^2=1$ and $\omega_4$ has two eigenspaces for eigenvalues $1$ and $-1$ of same dimension. We denote by $\Sigma^+$ and $\Sigma^-$ these subbundles. They decompose as follows:
$$
\left\{\begin{array}{ll}
\Sigma^+=(\Sigma^+M\otimes\Sigma^+E)\oplus(\Sigma^-M\otimes\Sigma^-E)\\ \\
\Sigma^-=(\Sigma^+M\otimes\Sigma^-E)\oplus(\Sigma^-M\otimes\Sigma^+E),
\end{array}
\right.
$$
where $\Sigma^{\pm}M$ and $\Sigma^{\pm}E$ are the spaces of half-spinors for $M$ and $E$ respectively. In the sequel, for $\varphi\in\Sigma,$ we will use the following convention:
$$\varphi=\varphi^{++}+\varphi^{--}+\varphi^{+-}+\varphi^{-+},$$
with 
$$\left\{
\begin{array}{l}
\varphi^{++}\in\Sigma^{++}:=\Sigma^+M\otimes\Sigma^+E,\\ 
\varphi^{--}\in\Sigma^{--}:=\Sigma^-M\otimes\Sigma^-E,\\
\varphi^{+-}\in\Sigma^{+-}:=\Sigma^+M\otimes\Sigma^-E,\\ 
\varphi^{-+}\in\Sigma^{-+}:=\Sigma^-M\otimes\Sigma^+E.
\end{array}\right.$$
Finally, we set
$$\varphi^+=\varphi^{++}+\varphi^{--}\quad\text{and}\quad\varphi^-=\varphi^{+-}+\varphi^{-+}.$$
If $\varphi$ is a Killing spinor of $\Mc,$ an easy computation yields
$$X|\varphi^+|^2=2\pre\langle \lambda X\cdot\varphi^{-},\varphi^{+}\rangle\hspace{1cm}\mbox{and}\hspace{1cm}X|\varphi^-|^2=2\pre\langle \lambda X\cdot\varphi^{+},\varphi^{-}\rangle.$$
\section{Main result}
\begin{thm}\label{thm1}\label{corollary first integration}
Let $(M^2,g)$ be an oriented Riemannian surface, with a given spin structure, and $E$ an oriented and spin vector bundle of rank $2$ on $M$. Let $\Sigma=\Sigma M\otimes\Sigma E$ be the twisted spinor bundle. Let $\lambda$ be a constant belonging to $\R\cup i\R$ and let $\vec{H}$ be a section of $E$. Let further $D$ be the Dirac operator of $\Sigma$. Then the three following statements are equivalent:
\begin{enumerate}
\item There exists a spinor $\varphi\in\Gamma(\Sigma)$ solution of the Dirac equation
\begin{equation}\label{dirac equation cor}
D\varphi=\vec{H}\cdot\varphi-2\lambda\varphi
\end{equation}
such that $\varphi^+$ and $\varphi^-$ do not vanish and satisfy
\begin{equation}\label{condphi+-}
X|\varphi^+|^2=2\pre\langle \lambda X\cdot\varphi^{-},\varphi^{+}\rangle\hspace{.5cm}\mbox{and}\hspace{.5cm}X|\varphi^-|^2=2\pre\langle \lambda X\cdot\varphi^{+},\varphi^{-}\rangle.
\end{equation}
\item There exists a spinor $\varphi\in\Gamma(\Sigma)$ solution of 
$$\nabla_X\varphi=-\frac{1}{2}\sum_{j=1,2}e_j\cdot B(X,e_j) \cdot\varphi+\lambda X\cdot\varphi, $$
where $B:TM\times TM\lgra E$ is bilinear and $\frac{1}{2}\trace(B)=\vec{H}$ and such that $\varphi^+$ and $\varphi^-$ do not vanish.
\item There exists a local isometric immersion of $(M,g)$ into $\Mc$ with normal bundle $E$, second fundamental form $B$ and mean curvature $\vec{H}$.
\end{enumerate}
The form $B$ and the spinor field $\varphi$ are linked by (\ref{B function phi}).
\end{thm}
In order to prove Theorem \ref{thm1} we consider the following equivalent technical 
\begin{prop}\label{proptothm1}
Let $M,$ $E$ and $\Sigma$ as in Theorem \ref{thm1}  and assume that there exists a spinor $\varphi\in\Gamma(\Sigma)$ solution of 
\begin{equation}\label{dirac equation}
D\varphi=\vec{H}\cdot\varphi-2\lambda\varphi
\end{equation} 
with $\varphi^+$ and $\varphi^-$ non-vanishing spinors satisfying (\ref{condphi+-}).
Then the symmetric bilinear map 
$$B:TM\times TM  \longrightarrow E$$
defined by
\begin{eqnarray}
\left<B(X,Y),\xi\right>\hspace{-.2cm}&=&\hspace{-.2cm}\frac{1}{2|\varphi^+|^2}\pre\left<X\cdot\nabla_Y\varphi^++Y\cdot\nabla_X\varphi^++2\lambda \langle X, Y\rangle\varphi^-,\xi\cdot\varphi^+\right>\nonumber\\
&&\hspace{-1cm}+\frac{1}{2|\varphi^-|^2}\pre\left<X\cdot\nabla_Y\varphi^-+Y\cdot\nabla_X\varphi^-+2\lambda \langle X, Y\rangle\varphi^+,\xi\cdot\varphi^-\right>\label{B function phi}
\end{eqnarray}
for all $X,Y\in \Gamma(TM)$ and all $\xi\in \Gamma(E)$ satisfies the Gauss, Codazzi and Ricci equations and is such that
$$\vec{H}=\frac{1}{2}\trace{B}.$$
\end{prop}
\noindent
\begin{rem}
If $\lambda=0,$ and if $\varphi\in\Gamma(\Sigma)$ is a solution of 
$$D\varphi=\vec{H}\cdot\varphi\hspace{1cm}\mbox{with}\hspace{1cm}|\varphi^+|=|\varphi^-|=1,$$
formula (\ref{B function phi}) simplifies to
\begin{eqnarray}
\left<B(X,Y),\xi\right>&=&\frac{1}{2}\pre\left<X\cdot\nabla_Y\varphi +Y\cdot\nabla_X\varphi,\xi\cdot\varphi\right>\nonumber\\
&=&\pre\left<X\cdot\nabla_Y\varphi,\xi\cdot\varphi\right>,\label{B phi c=0}
\end{eqnarray}
since this last expression is in fact symmetric in $X$ and $Y.$
\end{rem}
To prove proposition \ref{proptothm1} we first state the following lemma.
\begin{lem}\label{lem1}
Assume that $\varphi$ is a solution of the Dirac equation (\ref{dirac equation}) with $\varphi^+$ and $\varphi^-$ non-vanishing spinors satisfying (\ref{condphi+-}). Then, for all $X\in \Gamma(TM)$,
\begin{equation}\label{eqnnablaphi}
\nabla_X\varphi=\eta(X)\cdot\varphi+\lambda X\cdot\varphi,
\end{equation}
with 
\begin{equation}\label{relation eta B}
\eta(X)=-\frac{1}{2}\sum_{j=1}^2e_j\cdot B(e_j,X),
\end{equation}
where the bilinear map $B$ is defined by (\ref{B function phi}).
\end{lem}
\noindent
The proof of this lemma will be given in Section \ref{secpflem}. \\ \\
{\it Proof of Proposition \ref{proptothm1}:} The equations of Gauss, Codazzi and Ricci appear to be the integrability conditions of (\ref{eqnnablaphi}). Indeed  computing the spinorial curvature $\mathcal{R}$ for $\varphi$, we first observe that (\ref{relation eta B}) implies
$$X\cdot\eta(Y)-\eta(Y)\cdot X=B(X,Y)=Y\cdot\eta(X)-\eta(X)\cdot Y$$
for all $X,Y\in TM.$ 
Then, a direct computation yields
\beq\label{curv}
\mathcal{R}(X,Y)\varphi&=&d^{\nabla}\eta(X,Y)\cdot\varphi+\big(\eta(Y)\cdot\eta(X)-\eta(X)\cdot\eta(Y)\big)\cdot\varphi\\
&&+\lambda^2(Y\cdot X-X\cdot Y)\cdot\varphi,\nonumber
\eeq
where
$$d^{\nabla}\eta(X,Y)=\nabla_X(\eta(Y))-\nabla_Y(\eta(X))-\eta([X,Y]).$$
Here we also denote by $\nabla$ the natural connection on $Cl(TM\oplus E)=Cl(M)\hat\otimes Cl(E).$ 
\begin{lem}
We have:
\be
\item The left-hand side of \eqref{curv} satisfies
$$\mathcal{R}(e_1,e_2)\varphi=-\frac{1}{2}Ke_1\cdot e_2\cdot\varphi-\frac{1}{2}K_Ne_3\cdot e_4\cdot\varphi.$$
\item The first term of the right-hand side of \eqref{curv} satisfies
 $$d^{\nabla}\eta(X,Y)=-\frac{1}{2}\sum_{j=1}^2e_j\cdot\Big((\nabla^N_XB)(Y,e_j)-(\nabla^N_YB)(X,e_j)\Big)$$
 where $\nabla^N$ stands for the natural connection on $T^*M\otimes T^*M\otimes E.$
\item The second term of the right-hand side of \eqref{curv} satisfies
\beQ
\eta(e_2)\cdot\eta(e_1)-\eta(e_1)\cdot\eta(e_2)&=&\frac{1}{2}\big(|B(e_1,e_2)|^2-\left<B(e_1,e_1),B(e_2,e_2)\right>\big)e_1\cdot e_2\\
&&+\frac{1}{2}\left<\left(S_{e_3}\circ S_{e_4}-S_{e_4}\circ S_{e_3}\right)(e_1),e_2\right>e_3\cdot e_4.
\eeQ
\ee
\end{lem}
{\it Proof:} First, we compute $\mathcal{R}(e_1,e_2)\varphi$. We recall that $\Sigma=\Sigma M\otimes\Sigma E$ and suppose that $\varphi=\alpha\otimes\sigma$ with $\alpha\in\Sigma M$ and $\sigma\in\Sigma E$. Thus,
$$\mathcal{R}(e_1,e_2)\varphi=\mathcal{R}^M(e_1,e_2)\alpha\otimes\sigma+\alpha\otimes\mathcal{R}^E(e_1,e_2)\sigma,$$
where $\mathcal{R}^M$ and $\mathcal{R}^E$ are the spinorial curvatures on $M$ and $E$ respectively.
Moreover, by the Ricci identity on $M$, we have 
$$\mathcal{R}^M(e_1,e_2)\alpha=-\frac{1}{2}Ke_1\cdot e_2\cdot\alpha,$$
where $K$ is the Gauss curvature of $(M,g)$. Similarly, we have
$$\mathcal{R}^E(e_1,e_2)\sigma=-\frac{1}{2}K_Ne_3\cdot e_4\cdot\sigma,$$
where $K_N$ is the curvature of the connection on $E.$ These last two relations give the first point of the lemma.\\
For the second point of the lemma, we choose $e_j$ so that at $p\in M$, $\nabla {e_j}_{|p}=0$. Then, we have
\beQ
d^{\nabla}\eta(X,Y)&=&\nabla_X(\eta(Y))-\nabla_Y(\eta(X))-\eta([X,Y])\\
&=&\sum_{j=1,2}-\frac{1}{2}\nabla_X\big(e_j\cdot B(Y,e_j)\big)+\frac{1}{2}\nabla_Y\big(e_j\cdot B(X,e_j)\big)+\frac{1}{2}e_j\cdot B([X,Y],e_j)\\
&=&\sum_{j=1,2}-\frac{1}{2}e_j\cdot\nabla^E_X(B(Y,e_j))+\frac{1}{2}e_j\cdot\nabla^E_Y(B(X,e_j))+\frac{1}{2}e_j\cdot B(\nabla_XY,e_j)\\&&\hspace{.8cm}-\frac{1}{2}e_j\cdot B(\nabla_YX,e_j)\\
&=&-\frac{1}{2}\sum_{j=1,2}e_j\cdot\Big((\nabla^N_XB)(Y,e_j)-(\nabla^N_YB)(X,e_j)\Big)
\eeQ
since $[X,Y]=\nabla_XY-\nabla_YX$ and $(\nabla^N_XB)(Y,e_j)=\nabla^E_X(B(Y,e_j))-B(\nabla_XY,e_j).$ Here $\nabla^E$ stands for the given connection on $E$. 
\\We finally prove the third assertion of the lemma. In order to simplify the notation, we set $B(e_i,e_j)=B_{ij}$. We have
\beQ
\eta(e_2)\cdot\eta(e_1)-\eta(e_1)\cdot\eta(e_2)&=&-\frac{1}{4}\sum_{j,k=1}^2e_j\cdot B_{1j}\cdot e_k\cdot B_{2k}+\frac{1}{4}\sum_{j,k=1}^2e_j\cdot B_{2j}\cdot e_k\cdot B_{1k}\eeQ
\beQ
&=&\frac{1}{4}\Bigg[-e_1\cdot B_{11}\cdot e_1\cdot B_{21}-e_1\cdot B_{11}\cdot e_2\cdot B_{22}-e_2\cdot B_{12}\cdot e_1\cdot B_{21}-e_2\cdot B_{12}\cdot e_2\cdot B_{22}\\
&&+e_1\cdot B_{21}\cdot e_1\cdot B_{11}+e_1\cdot B_{21}\cdot e_2\cdot B_{12}+e_2\cdot B_{22}\cdot e_1\cdot B_{11}+e_2\cdot B_{22}\cdot e_2\cdot B_{12}\Bigg]\\
&=&\frac{1}{2}\Big[|B_{12}|^2-\left<B_{11},B_{22}\right>\Big]e_1\cdot e_2+\frac{1}{4}\Big[-B_{11}\cdot B_{21}+B_{21}\cdot B_{11}-B_{12}\cdot B_{22}+B_{22}\cdot B_{12}\Big].
\eeQ
Now, if we write $B_{ij}=B_{ij}^3e_3+B_{ij}^4e_4$, we have
$$-B_{11}\cdot B_{21}+B_{21}\cdot B_{11}=2(-B_{11}^3B_{21}^4+B_{21}^3B_{11}^4)e_3\cdot e_4$$
and
$$-B_{12}\cdot B_{22}+B_{22}\cdot B_{12}=2(-B_{12}^3B_{22}^4+B_{12}^4B_{22}^3)e_3\cdot e_4.$$
Moreover
$$-B_{11}^3B_{21}^4+B_{21}^3B_{11}^4-B_{12}^3B_{22}^4+B_{12}^4B_{22}^3=\left<(S_{e_3}\circ S_{e_4}-S_{e_4}\circ S_{e_3})(e_1),e_2\right>$$
since for $j\in\{1,2\}$ and $k\in\{3,4\}$, we have $S_{e_k}e_j=B_{j1}^ke_1+B_{j2}^ke_2.$ The formula follows.
$\finpreuve$\\
Now, we give this final lemma
\begin{lem}\label{order2}
If $T$ is an element of $\mathcal{C}l(M)\hat{\otimes}\mathcal{C}l(E)$ of order $2$, that is of
$$\Lambda^2M\otimes1\ \oplus\ TM\otimes E\ \oplus\ 1\otimes\Lambda^2E,$$
so that 
$$T\cdot\varphi=0,$$
where $\varphi$ is a spinor field of $\Sigma$ such that $\varphi^+$ and $\varphi^-$ do not vanish, then $T=0$.
\end{lem}
\noindent
{\it Proof:} We have
$$\mathcal{C}l_2\hat{\otimes}\mathcal{C}l_2\simeq\mathcal{C}l_4\simeq\HH(2),$$
where $\HH(2)$ is the set of $2\times2$ matrices with quaternionic coefficients. The spinor bundle $\Sigma$ and the Clifford product come from the representation
$$\HH(2)\lgra\eend_{\HH}(\HH\oplus\HH).$$
The first factor of $\HH\oplus\HH$ correspond to $\Sigma^+$ and the second to $\Sigma^-$. Moreover, elements of order $2$ of $\mathcal{C}l_4$ are matrices 
$$\left(\begin{array}{cc}
p&0\\0&q\end{array}\right),$$
where $p,q$ are purely imaginary quaternions. Hence $T\cdot\varphi=0$ is equivalent to
$$\left(\begin{array}{cc}
p&0\\0&q\end{array}\right)
\left(\begin{array}{c}
\alpha\\ \sigma\end{array}\right)
=\left(\begin{array}{cc}
0\\0\end{array}\right)$$
with $\alpha,\sigma$ non zero quaternions. Thus $p=q=0$, and so $T$ vanishes identically.
$\finpreuve$\\
We deduce from \eqref{curv} and Lemma \ref{order2} and comparing terms, that
$$\left\{
\begin{array}{l}
K=\left<B(e_1,e_1),B(e_2,e_2)\right>-|B(e_1,e_2)|^2+4\lambda^2,\\ \\
K_N=-\left<(S_{e_3}\circ S_{e_4}-S_{e_4}\circ S_{e_3})(e_1),e_2\right>,\\ \\
(\widetilde{\nabla}_XB)(Y,e_j)-(\widetilde{\nabla}_YB)(X,e_j)=0,\quad\forall j=1,2,
\end{array}\right.$$
which are respectively the Gauss, Ricci and Codazzi equations. 
$\finpreuve$\\
From Proposition \ref{proptothm1} and the fundamental theorem of submanifolds, we deduce that a spinor field solution of (\ref{dirac equation}) such that (\ref{condphi+-}) holds defines a local isometric immersion of $M$ into $\Mc$ with normal bundle $E$ and second fundamental form $B$. This implies the equivalence between assertions 1 and 3 in  Theorem \ref{thm1}. The equivalence between assertions 1 and 2 is given by Lemma \ref{lem1} and will be proven in the next section. 
\begin{rem} If in Theorem \ref{thm1} we assume moreover that the manifold is simply connected, the spinor field solution of (\ref{dirac equation}) defines a global isometric immersion of  $M$ into $\Mc$.
\end{rem}
\section{Proof of Lemma \ref{lem1}}\label{secpflem}
In order to prove Lemma \ref{lem1}, we need some preliminary results. First, we remark that 
$$\left\{\begin{array}{l}
\ D\varphi^{--}=\vec{H}\cdot\varphi^{++}-2\lambda\varphi^{+-},\\
\ D\varphi^{++}=\vec{H}\cdot\varphi^{--}-2\lambda\varphi^{-+},\\
\ D\varphi^{+-}=\vec{H}\cdot\varphi^{-+}-2\lambda\varphi^{--},\\
 \ D\varphi^{-+}=\vec{H}\cdot\varphi^{+-}-2\lambda\varphi^{++}.
 \end{array}\right.$$
We fix a point $p\in M,$ and consider $e_3$ a unit vector in $E_p$ so that $\vec{H}=|\vec H|e_3$ at $p.$ We complete $e_3$ by $e_4$ to get a positively oriented and orthonormal frame of $E_p$. We first assume that $\varphi^{--},$ $\varphi^{++},$ $\varphi^{+-}$ and $\varphi^{-+}$ do not vanish at $p.$ We see easily that 
$$\left\{e_1\cdot e_3\cdot\frac{\varphi^{--}}{|\varphi^{--}|},e_2\cdot e_3\cdot\frac{\varphi^{--}}{|\varphi^{--}|}\right\}$$
is an orthonormal frame of $\Sigma^{++}$ for the real scalar product $\pre\left<\cdot,\cdot\right>$. Indeed, we have
\beQ
\pre\left<e_1\cdot e_3\cdot\varphi^{--},e_2\cdot e_3\cdot\varphi^{--}\right>&=&\pre\left<\varphi^{--},e_3\cdot e_1\cdot e_2\cdot e_3\cdot\varphi^{--}\right>\\
&=&\pre\left(i|\varphi^{--}|^2\right)=0.
\eeQ
Analogously, 
$$\left\{e_1\cdot e_3\cdot\frac{\varphi^{++}}{|\varphi^{++}|},e_2\cdot e_3\cdot\frac{\varphi^{++}}{|\varphi^{++}|}\right\},$$
$$\left\{e_1\cdot e_3\cdot\frac{\varphi^{-+}}{|\varphi^{-+}|},e_2\cdot e_3\cdot\frac{\varphi^{-+}}{|\varphi^{-+}|}\right\},$$
$$\left\{e_1\cdot e_3\cdot\frac{\varphi^{+-}}{|\varphi^{+-}|},e_2\cdot e_3\cdot\frac{\varphi^{+-}}{|\varphi^{+-}|}\right\}$$
are orthonormal frames of $\Sigma^{--}$, $\Sigma^{+-}$ and $\Sigma^{-+}$ respectively. We define the following bilinear forms
$$F_{++}(X,Y)=\pre\left<\nabla_X\varphi^{++},Y\cdot e_3\cdot\varphi^{--}\right>,$$
$$F_{--}(X,Y)=\pre\left<\nabla_X\varphi^{--},Y\cdot e_3\cdot\varphi^{++}\right>,$$
$$F_{+-}(X,Y)=\pre\left<\nabla_X\varphi^{+-},Y\cdot e_3\cdot\varphi^{-+}\right>,$$
$$F_{-+}(X,Y)=\pre\left<\nabla_X\varphi^{-+},Y\cdot e_3\cdot\varphi^{+-}\right>,$$
and 
$$B_{++}(X,Y)=-\pre\left<\lambda X\cdot\varphi^{-+},Y\cdot e_3\cdot\varphi^{--}\right>,$$
$$B_{--}(X,Y)=-\pre\left<\lambda X\cdot\varphi^{+-},Y\cdot e_3\cdot\varphi^{++}\right>,$$
$$B_{+-}(X,Y)=-\pre\left<\lambda X\cdot\varphi^{--},Y\cdot e_3\cdot\varphi^{-+}\right>,$$
$$B_{-+}(X,Y)=-\pre\left<\lambda X\cdot\varphi^{++},Y\cdot e_3\cdot\varphi^{+-}\right>.$$
We have this first lemma: 
\begin{lem}\label{first lemma tr}
We have
\begin{enumerate} 
\item $\thmrm{tr}(F_{++})=-|\vec{H}||\varphi^{--}|^2+2\pre\left<\lambda\varphi^{-+},e_3\cdot\varphi^{--}\right>,$
\item $\thmrm{tr}(F_{--})=-|\vec{H}||\varphi^{++}|^2+2\pre\left<\lambda\varphi^{+-},e_3\cdot\varphi^{++}\right>,$
\item $\thmrm{tr}(F_{+-})=-|\vec{H}||\varphi^{-+}|^2+2\pre\left<\lambda\varphi^{--},e_3\cdot\varphi^{-+}\right>,$
\item $\thmrm{tr}(F_{-+})=-|\vec{H}||\varphi^{+-}|^2+2\pre\left<\lambda\varphi^{++},e_3\cdot\varphi^{+-}\right>.$
\end{enumerate}
\end{lem}
\noindent
This second lemma gives the defect of symmetry:
\begin{lem}\label{lemsym1} We have
\begin{enumerate}
\item $F_{++}(e_1,e_2)=F_{++}(e_2,e_1)-2\pre\left<\lambda\varphi^{-+},e_1\cdot e_2\cdot e_3\cdot\varphi^{--}\right>,$
\item $F_{--}(e_1,e_2)=F_{--}(e_2,e_1)-2\pre\left<\lambda\varphi^{+-},e_1\cdot e_2\cdot e_3\cdot\varphi^{++}\right>,$
\item $F_{+-}(e_1,e_2)=F_{+-}(e_2,e_1)-2\pre\left<\lambda\varphi^{--},e_1\cdot e_2\cdot e_3\cdot\varphi^{-+}\right>,$
\item $F_{-+}(e_1,e_2)=F_{-+}(e_2,e_1)-2\pre\left<\lambda\varphi^{++},e_1\cdot e_2\cdot e_3\cdot\varphi^{+-}\right>.$
\end{enumerate}
\end{lem}
For sake of brevity, we only prove Lemma \ref{lemsym1}. The proof of Lemma \ref{first lemma tr} is very similar.
\\
\\{\it Proof:}  We have
\beQ
F_{++}(e_1,e_2)&=&\pre\left<\nabla_{e_1}\varphi^{++},e_2\cdot e_3\cdot\varphi^{--}\right>\\
&=&\pre\left<e_1\cdot\nabla_{e_1}\varphi^{++},e_1\cdot e_2\cdot e_3\cdot\varphi^{--}\right>\\
&=&\pre\left<\vec{H}\cdot\varphi^{--} -2\lambda\varphi^{-+}-e_2\cdot\nabla_{e_2}\varphi^{++},e_1\cdot e_2\cdot e_3\cdot\varphi^{--}\right>.
\eeQ
The first term is
\beQ
\pre\left<\vec{H}\cdot\varphi^{--},e_1\cdot e_2\cdot e_3\cdot\varphi^{--}\right>&=&-\pre\left<\varphi^{--},i\vec{H}\cdot e_3\cdot\varphi^{--}\right>\\
&=&-\pre\left(i|\vec{H}||\varphi^{--}|^2\right)=0.
\eeQ
Hence, we get
\beQ
F_{++}(e_1,e_2)+2\pre\left<\lambda\varphi^{-+},e_1\cdot e_2\cdot e_3\cdot\varphi^{--}\right>&=&-\pre\left<e_2\cdot\nabla_{e_2}\varphi^{++},e_1\cdot e_2\cdot e_3\cdot\varphi^{--}\right>\\
&=&\pre\left<\nabla_{e_2}\varphi^{++},e_2\cdot e_1\cdot e_2\cdot e_3\cdot\varphi^{--}\right>\\
&=&\pre\left<\nabla_{e_2}\varphi^{++},e_1\cdot e_3\cdot\varphi^{--}\right>\\
&=&F_{++}(e_2,e_1).
\eeQ
The proof is similar for the three other forms.
$\finpreuve$\\

By analogous computations, we also get the following lemmas:
\begin{lem} We have
\begin{enumerate}
\item $tr(B_{++})=-2\pre\left<\lambda\varphi^{-+},e_3\cdot\varphi^{--}\right>,$
\item $tr(B_{--})=-2\pre\left<\lambda\varphi^{+-},e_3\cdot\varphi^{++}\right>,$
\item $tr(B_{+-})=-2\pre\left<\lambda\varphi^{--},e_3\cdot\varphi^{-+}\right>,$
\item $tr(B_{-+})=-2\pre\left<\lambda\varphi^{++},e_3\cdot\varphi^{+-}\right>.$
\end{enumerate}
\end{lem}

\begin{lem} We have
\begin{enumerate}\label{lemsym2}
\item $B_{++}(e_1,e_2)=B_{++}(e_2,e_1)+2\pre\left<\lambda\varphi^{-+},e_1\cdot e_2\cdot e_3\cdot\varphi^{--}\right>,$
\item $B_{--}(e_1,e_2)=B_{--}(e_2,e_1)+2\pre\left<\lambda\varphi^{+-},e_1\cdot e_2\cdot e_3\cdot\varphi^{++}\right>,$
\item $B_{+-}(e_1,e_2)=B_{+-}(e_2,e_1)+2\pre\left<\lambda\varphi^{--},e_1\cdot e_2\cdot e_3\cdot\varphi^{-+}\right>,$
\item $B_{-+}(e_1,e_2)=B_{-+}(e_2,e_1)+2\pre\left<\lambda\varphi^{++},e_1\cdot e_2\cdot e_3\cdot\varphi^{+-}\right>.$
\end{enumerate}
\end{lem}

Now, we set $$\left\{\begin{array}{l}A_{++}:=F_{++}+B_{++},\\
A_{--}:=F_{--}+B_{--},\\A_{+-}:=F_{+-}+B_{+-},\\
A_{-+}:=F_{-+}+B_{-+},
\end{array}\right.$$
and 
$$F_+=\dfrac{A_{++}}{|\varphi^{--}|^2}-\dfrac{A_{--}}{|\varphi^{++}|^2}\hspace{1cm}\mbox{and}\hspace{1cm}F_-=\dfrac{A_{+-}}{|\varphi^{-+}|^2}-\dfrac{A_{-+}}{|\varphi^{+-}|^2}.$$
From the last four lemmas we deduce immediately that  $F_+$ and $F_-$ are symmetric and trace-free. Moreover, by a direct computation using the conditions (\ref{condphi+-}) on the norms of $\varphi^+$ and $\varphi^-$, we get the following lemma:

\begin{lem}\label{lem F=0}
 The symmetric operators $F^+$ and $F^-$ of $TM$ associated to the bilinear forms $F_+$ and $F_-,$ defined by
$$F^+(X)=F_+(X,e_1)e_1+F_+(X,e_2)e_2\hspace{.5cm}\mbox{and}\hspace{.5cm}F^-(X)=F_-(X,e_1)e_1+F_-(X,e_2)e_2$$
for all $X\in TM,$ satisfy
\begin{enumerate}
\item $\pre\left<F^+(X)\cdot e_3\cdot\varphi^{--},\varphi^{++}\right>=0,$
\item $\pre\left<F^-(X)\cdot e_3\cdot\varphi^{-+},\varphi^{+-}\right>=0.$
\end{enumerate}
\end{lem}

\begin{proof}
Since 
$$A_{++}(X,Y)=\Re e\ \langle\nabla_X\varphi^{++}-\lambda X\cdot\varphi^{-+},Y\cdot e_3\cdot\varphi^{--}\rangle,$$ 
and since $(e_1\cdot e_3\cdot\frac{\varphi^{--}}{|\varphi^{--}|},e_2\cdot e_3\cdot\frac{\varphi^{--}}{|\varphi^{--}|})$ is an orthonormal frame of $\Sigma^{++},$ we have
\begin{eqnarray*}
&&\Re e\ \langle\nabla_X\varphi^{++}-\lambda X\cdot\varphi^{-+},\varphi^{++}\rangle\\= &&
\frac{A_{++}}{|\varphi^{--}|^2}(X,e_1)\ \Re e\ \langle e_1\cdot e_3\cdot\varphi^{--},\varphi^{++}\rangle+\frac{A_{++}}{|\varphi^{--}|^2}(X,e_2)\ \Re e\ \langle e_2\cdot e_3\cdot\varphi^{--},\varphi^{++}\rangle.
\end{eqnarray*}
Similarly,
\begin{eqnarray*}
&&\Re e\ \langle\nabla_X\varphi^{--}-\lambda X\cdot\varphi^{+-},\varphi^{--}\rangle\\ &&=
\frac{A_{--}}{|\varphi^{++}|^2}(X,e_1)\ \Re e\ \langle e_1\cdot e_3\cdot\varphi^{++},\varphi^{--}\rangle+\frac{A_{--}}{|\varphi^{++}|^2}(X,e_2)\ \Re e\ \langle e_2\cdot e_3\cdot\varphi^{++},\varphi^{--}\rangle
\\ && =-
\frac{A_{--}}{|\varphi^{++}|^2}(X,e_1)\ \Re e\ \langle e_1\cdot e_3\cdot\varphi^{--},\varphi^{++}\rangle-\frac{A_{--}}{|\varphi^{++}|^2}(X,e_2)\ \Re e\ \langle e_2\cdot e_3\cdot\varphi^{--},\varphi^{++}\rangle.
\end{eqnarray*}
These two formulas imply that
$$\pre\left<F^+(X)\cdot e_3\cdot\varphi^{--},\varphi^{++}\right>= \Re e\ \langle\nabla_X\varphi^{+}-\lambda X\cdot\varphi^{-},\varphi^{+}\rangle;$$
by the first condition in (\ref{condphi+-}), this last expression is zero.
\end{proof}
Hence, the operators $F^+$ and $F^-$ are of rank at most $\leq 1.$ Since they are symmetric and trace-free, they vanish identically.\\\\
Using again that $(e_1\cdot e_3\cdot\frac{\varphi^{--}}{|\varphi^{--}|},e_2\cdot e_3\cdot\frac{\varphi^{--}}{|\varphi^{--}|})$ is an orthonormal frame of $\Sigma^{++},$ we have
$$\nabla_X\varphi^{++}=F_{++}(X,e_1)e_1\cdot e_3\cdot\frac{\varphi^{--}}{|\varphi^{--}|}+F_{++}(X,e_2)e_2\cdot e_3\cdot\frac{\varphi^{--}}{|\varphi^{--}|}.$$
Since $F_{++}=A_{++}-B_{++}$ and denoting by $A^{++}$ and $B^{++}$ the symmetric operators of $TM$ associated to $A_{++}$ and $B_{++}$ and defined by
$$A^{++}(X)=A_{++}(X,e_1)e_1+A_{++}(X,e_2)e_2,\ B^{++}(X)=B_{++}(X,e_1)e_1+B_{++}(X,e_2)e_2,$$
we get
\begin{equation}\label{nabla phi ++}
\nabla_X\varphi^{++}=\frac{1}{|\varphi^{--}|^2}\left[A^{++}(X)\cdot e_3\cdot\varphi^{--}-B^{++}(X)\cdot e_3\cdot\varphi^{--}\right].
\end{equation}
Similarly, if $A^{--}$ and $B^{--}$ denote the symmetric operators of $TM$ associated to $A_{--}$ and $B_{--},$ we have
\begin{equation}\label{nabla phi --}
\nabla_X\varphi^{--}=\frac{1}{|\varphi^{++}|^2}\left[A^{--}(X)\cdot e_3\cdot\varphi^{++}-B^{--}(X)\cdot e_3\cdot\varphi^{++}\right].
\end{equation}
Moreover, we easily get
$$B^{++}(X)\cdot e_3\cdot\varphi^{--}=-|\varphi^{--}|^2\lambda X\cdot\varphi^{-+}\hspace{.2cm}\mbox{and}\hspace{.2cm}B^{--}(X)\cdot e_3\cdot\varphi^{++}=-|\varphi^{++}|^2\lambda X\cdot\varphi^{+-}.$$
Thus
\begin{eqnarray*}
\nabla_X\varphi^{+}&=&\frac{1}{|\varphi^{--}|^2}A^{++}(X)\cdot e_3\cdot\varphi^{--}+\lambda X\cdot\varphi^{-+}\\
&&+\ \frac{1}{|\varphi^{++}|^2}A^{--}(X)\cdot e_3\cdot\varphi^{++}+\lambda X\cdot\varphi^{+-}.
\end{eqnarray*}
Setting $A^+=A^{++}+A^{--}$ we get from the definition of $A^{++}$ and $A^{--}$ and from  $F^+=0$ that $\frac{A^{++}}{|\varphi^{--}|^2}=\frac{A^{--}}{|\varphi^{++}|^2}$. Bearing in mind that $|\varphi^+|^2=|\varphi^{++}|^2+|\varphi^{--}|^2$, we get finally
\begin{equation}\label{A+}
\frac{A^+}{|\varphi^+|^2}=\frac{A^{++}}{|\varphi^{--}|^2}=\frac{A^{--}}{|\varphi^{++}|^2}.
 \end{equation} 
 Thus 
 \begin{equation}\label{nabla phi +}
 \nabla_X\varphi^{+}=\frac{1}{|\varphi^+|^2}A^+(X)\cdot e_3\cdot\varphi^++\lambda X\cdot\varphi^-.
  \end{equation} 
Similarly, denoting by $A^{+-}$ and $A^{-+}$ the symmetric operators of $TM$ associated to $A_{+-}$ and $A_{-+},$ setting $A^-=A^{+-}+A^{-+}$ and using $F^-=0$ we get
\begin{eqnarray}
\nabla_X\varphi^-&=&\frac{1}{|\varphi^{+-}|^2}A^{-+}(X)\cdot e_3\cdot\varphi^{+-}+\lambda X\cdot\varphi^{++}\nonumber\\
&&+\ \frac{1}{|\varphi^{-+}|^2}A^{+-}(X)\cdot e_3\cdot\varphi^{-+}+\lambda X\cdot\varphi^{--}\nonumber\\
&=&\frac{1}{|\varphi^-|^2}A^-(X)\cdot e_3\cdot\varphi^-+\lambda X\cdot\varphi^+.\label{nabla phi -}
\end{eqnarray}
We now observe that formulas (\ref{nabla phi +}) and (\ref{nabla phi -}) also hold if $\varphi^{++}$ or $\varphi^{--},$ (resp. $\varphi^{+-}$ or $\varphi^{-+}$) vanishes at $p:$ indeed, assuming for instance that $\varphi^{++}(p)=0,$ and thus that $\varphi^{--}(p)\neq 0$ since $\varphi^+(p)\neq 0,$ equation (\ref{nabla phi ++}) holds, and, from the first condition in (\ref{condphi+-}),
$$\Re e\ \langle\nabla_X\varphi^{--}-\lambda X\cdot\varphi^{+-},\varphi^{--}\rangle=0.$$
Since $\left(\frac{\varphi^{--}}{|\varphi^{--}|},i\frac{\varphi^{--}}{|\varphi^{--}|}\right)$ is an orthonormal basis of $\Sigma^{--},$ we deduce that
$$\nabla_X\varphi^{--}-\lambda X\cdot\varphi^{+-}=i\alpha(X)\frac{\varphi^{--}}{|\varphi^{--}|}$$
for some real 1-form $\alpha.$ Since $D\varphi^{--}+2\lambda\varphi^{+-}=0$ ($\varphi^{++}=0$ at $p$),  this implies that
$$\left(\alpha(e_1)e_1+\alpha(e_2)e_2\right)\cdot \frac{\varphi^{--}}{|\varphi^{--}|}=0,$$
and thus that $\alpha=0.$ We thus get $\nabla_X\varphi^{--}=\lambda X\cdot\varphi^{+-}$ instead of (\ref{nabla phi --}), which, together with (\ref{nabla phi ++}), easily implies (\ref{nabla phi +}).
\\

Now, we set
$$\eta^+(X)=\left(\frac{1}{|\varphi^+|^2}A^+(X)\cdot e_3\right)^+\hspace{.5cm}\mbox{and}\hspace{.5cm}\eta^-(X)=\left(\frac{1}{|\varphi^-|^2}A^-(X)\cdot e_3\right)^-$$
where,  if $\sigma$ belongs to $\mathcal{C}l^0(TM\oplus E)$, we denote by $\sigma^+:=\frac{1+\omega_4}{2}\cdot\sigma$ and by $\sigma^-:=\frac{1-\omega_4}{2}\cdot\sigma$ the parts of $\sigma$ acting on $\Sigma^+$ and on $\Sigma^-$ only, i.e., such that
$$\sigma^+\cdot\varphi=\sigma\cdot\varphi^+\ \in\ \Sigma^+\hspace{.5cm}\mbox{and}\hspace{.5cm}\sigma^-\cdot\varphi=\sigma\cdot\varphi^-\ \in\ \Sigma^-.$$
Setting $\eta=\eta^++\eta^-$ we thus get
$$\nabla_X\varphi=\eta(X)\cdot\varphi+\lambda X\cdot\varphi,$$
as claimed in Lemma \ref{lem1}.

Explicitely, setting $A_+(X,Y):=\langle A^+(X),Y\rangle$ and $A_-(X,Y):=\langle A^-(X),Y\rangle,$ the form $\eta$ is given by
\begin{eqnarray*}
\eta(X)&=&\frac{1}{2|\varphi^+|^2}\left[A_+(X,e_1)(e_1\cdot e_3-e_2\cdot e_4)+A_+(X,e_2)(e_2\cdot e_3+e_1\cdot e_4)\right]\\
&&+\frac{1}{2|\varphi^-|^2}\left[A_{-}(X,e_1)(e_1\cdot e_3+e_2\cdot e_4)+A_-(X,e_2)(e_2\cdot e_3-e_1\cdot e_4)\right]
\end{eqnarray*}
with 
$$A_+(X,Y)=\pre \left<\nabla_X\varphi^+-\lambda X\cdot\varphi^-,Y\cdot e_3\cdot\varphi^+\right>$$
and
$$A_-(X,Y)=\pre \left<\nabla_X\varphi^--\lambda X\cdot\varphi^+,Y\cdot e_3\cdot\varphi^-\right>.$$
By direct computations we get that 
$$B(X,Y):=X\cdot \eta(Y)-\eta(Y)\cdot X$$
is a vector belonging to $E$ which is such that
\begin{eqnarray*}
\left<B(X,Y),\xi\right>&=&\frac{1}{|\varphi^+|^2}\pre\left<X\cdot\nabla_Y\varphi^+-\lambda X\cdot Y\cdot\varphi^-,\xi\cdot\varphi^+\right>\\
&&+\frac{1}{|\varphi^-|^2}\pre\left<X\cdot\nabla_Y\varphi^--\lambda X\cdot Y\cdot\varphi^+,\xi\cdot\varphi^-\right>
\end{eqnarray*}
for all $\xi\in E.$ This last expression appears to be symmetric in $X,Y$ (the proof is analogous to the proof of the symmetry of $A_{++}=F_{++}+B_{++}$ above). Computing 
$$\left<B(X,Y),\xi\right>=\frac{1}{2}\left(\left<B(X,Y),\xi\right>+\left<B(Y,X),\xi\right>\right)$$
we finally obtain that $B$ is given by formula (\ref{B function phi}).

Since $B(e_j,X)=e_j\cdot\eta(X)-\eta(X)\cdot e_j,$ we obtain
\begin{equation*}
\sum_{j=1,2} e_j\cdot B(e_j,X)=-2\eta(X)-\sum_{j=1,2} e_j\cdot\eta(X)\cdot e_j.
\end{equation*}
Writing $\eta(X)$ in the form $\sum_k e_k\cdot n_k$ for some vectors $n_k$ belonging to $E$, we easily get that $\sum_j e_j\cdot\eta(X)\cdot e_j=0$. Thus
$$\eta(X)=-\frac{1}{2}\sum_{j=1,2} e_j\cdot B(e_j,X).$$
The last claim in Lemma \ref{lem1} is proved.
$\finpreuve$\\
\section{Weierstrass representation of surfaces in $\R^4$}
We are interested here in isometric immersions in euclidean space $\R^4$ (thus $c=\lambda=0$); we obtain that the immersions are given by a formula which generalizes the representation formula given by T. Friedrich in \cite{Fr}. Such a formula was also found in \cite{Fr0} using a different method involving twistor theory.  
\\

We consider the scalar product $\langle\langle.,.\rangle\rangle$ defined on $\Sigma^+$ by
$$\langle\langle.,.\rangle\rangle:\begin{array}[t]{rcl}\Sigma^+\times\Sigma^+&\rightarrow&\HH\\(\varphi^+,\psi^+)&\mapsto&\overline{[\psi^+]}[\varphi^+],\end{array}$$
where $[\varphi^+]$ and $[\psi^+]$ $\in\HH$ represent the spinors $\varphi^+$ and $\psi^+$ in some frame, and where, if $q=q_1\it{1}+q_2I+q_3J+q_4K$ belongs to $\HH,$ 
$$\overline{q}=q_1\it{1}-q_2I-q_3J-q_4K.$$ 
We also define the product $\langle\langle.,.\rangle\rangle$ on $\Sigma^-$ by an analogous formula:
$$\langle\langle.,.\rangle\rangle:\begin{array}[t]{rcl}\Sigma^-\times\Sigma^-&\rightarrow&\HH\\(\varphi^-,\psi^-)&\mapsto&\overline{[\psi^-]}[\varphi^-].\end{array}$$
The following properties hold: for all $\varphi,\psi\in\Sigma$ and all $X\in TM\oplus E,$
\begin{equation}\label{pty pairing1}
\langle\langle\varphi^+,\psi^+\rangle\rangle=\overline{\langle\langle\psi^+,\varphi^+\rangle\rangle},\hspace{1cm}\langle\langle\varphi^-,\psi^-\rangle\rangle=\overline{\langle\langle\psi^-,\varphi^-\rangle\rangle}
\end{equation}
and
\begin{equation}\label{pty pairing2}
\langle\langle X\cdot \varphi^+,\psi^-\rangle\rangle=-\langle\langle\varphi^+,X\cdot \psi^-\rangle\rangle.
\end{equation}
Assume that we have a spinor $\varphi$ solution of the Dirac equation $D\varphi=\vec{H}\cdot\varphi$ so that $|\varphi^+|=|\varphi^-|=1,$ and define the $\HH$-valued $1$-form $\xi$ by
$$\xi(X)=\langle\langle X\cdot\varphi^-,\varphi^+\rangle\rangle\hspace{1cm}\in\hspace{.5cm}\HH.$$
\begin{prop}
The form $\xi\in\Omega^1(M,\HH)$ is closed.
\end{prop}
\noindent
{\it Proof:} By a straightforward computation, we get
$$d\xi(e_1,e_2)=\langle\langle e_2\cdot\nabla_{e_1}\varphi^-,\varphi^+\rangle\rangle-\langle\langle e_1\cdot\nabla_{e_2}\varphi^-,\varphi^+\rangle\rangle+\langle\langle e_2\cdot\varphi^-,\nabla_{e_1}\varphi^+\rangle\rangle-\langle\langle e_1\cdot\varphi^-,\nabla_{e_2}\varphi^+\rangle\rangle.$$
First observe that 
\begin{eqnarray*}
\langle\langle e_2\cdot\nabla_{e_1}\varphi^-,\varphi^+\rangle\rangle-\langle\langle e_1\cdot\nabla_{e_2}\varphi^-,\varphi^+\rangle\rangle&=&-\langle\langle e_1\cdot\nabla_{e_1}\varphi^-,e_1\cdot e_2\cdot \varphi^+\rangle\rangle+\langle\langle e_2\cdot\nabla_{e_2}\varphi^-,e_2\cdot e_1\cdot\varphi^+\rangle\rangle\\
&=&-\langle\langle D\varphi^-,e_1\cdot e_2\cdot\varphi^+\rangle\rangle
\end{eqnarray*}
and similarly that 
\begin{eqnarray*}
\langle\langle e_2\cdot\varphi^-,\nabla_{e_1}\varphi^+\rangle\rangle-\langle\langle e_1\cdot\varphi^-,\nabla_{e_2}\varphi^+\rangle\rangle&=& \langle\langle e_1\cdot e_2\cdot\varphi^-,e_1\cdot\nabla_{e_1}\varphi^+\rangle\rangle-\langle\langle e_2\cdot e_1\cdot\varphi^-,e_2\cdot\nabla_{e_2}\varphi^+\rangle\rangle\\
&=&\langle\langle e_1\cdot e_2\cdot\varphi^-,D\varphi^+\rangle\rangle.
\end{eqnarray*}
Thus
$$d\xi(e_1,e_2)=\langle\langle e_1\cdot e_2\cdot D\varphi^-,\varphi^+\rangle\rangle+\langle\langle e_1\cdot e_2\cdot \varphi^-,D\varphi^+\rangle\rangle.$$
Since $D\varphi=\vec{H}\cdot\varphi$, then $D\varphi^+=\vec{H}\cdot\varphi^+$ and $D\varphi^-=\vec{H}\cdot\varphi^-$, which implies
$$d\xi(e_1,e_2)=\langle\langle(e_1\cdot e_2\cdot\vec{H}-\vec{H}\cdot e_1\cdot e_2)\cdot\varphi^-,\varphi^+\rangle\rangle=0.$$
$\finpreuve$\\
Assuming that $M$ is simply connected, there exists a function $F:M\lgra\HH$ so that $dF=\xi$. We now identify $\HH$ to $\R^4$ in the natural way.
\begin{thm}\label{theorem second integration}
\begin{enumerate}
\item The map $F=(F_1,F_2,F_3,F_4):M\lgra\R^4$ is an isometry.
\item The map 
\function{\Phi_E}{E}{M\times\R^4}{X\in E_m}{(F(m),\xi_1(X),\xi_2(X),\xi_3(X),\xi_4(X))}
is an isometry between $E$ and the normal bundle $N(F(M))$ of $F(M)$ into $\R^4$, preserving connections and second fundamental forms.
\end{enumerate}
\end{thm}
\begin{proof}
Note first that the euclidean norm of $\xi\in \R^4\simeq\HH$ is 
$$|\xi|^2=\langle\xi,\xi\rangle=\overline{\xi}\xi\hspace{.5cm}\in \hspace{.5cm}\R,$$
and more generally that the real scalar product $\langle\xi,\xi'\rangle$ of  $\xi,\xi'\in \R^4\simeq\HH$ is the component of $\textit{1}$ in $\langle\langle\xi,\xi'\rangle\rangle=\overline{\xi'}\xi\in\HH.$ We first compute, for all $X,Y$ belonging to $E\cup TM,$
\begin{eqnarray*}
\overline{\xi(Y)}\xi(X)&=&\overline{\langle\langle Y\cdot\varphi^-,\varphi^+\rangle\rangle}\langle\langle X\cdot\varphi^-,\varphi^+\rangle\rangle=\overline{\left(\overline{[\varphi^+]}[Y\cdot\varphi^-]\right)}\left(\overline{[\varphi^+]}[X\cdot\varphi^-]\right)\\
&=&\overline{[Y\cdot\varphi^-]}[X\cdot\varphi^-]
\end{eqnarray*}
since $[\varphi^+]\overline{[\varphi^+]}=1$ ($|\varphi^+|=1$). Here and below the brackets $[.]$ stand for the components  ($\in \HH$) of the spinor fields in some local frame. Thus
\begin{equation}\label{formula xiX xiY}
\overline{\xi(Y)}\xi(X)=\langle\langle X\cdot\varphi^-,Y\cdot\varphi^-\rangle\rangle,
\end{equation}
which in particular implies (considering the components of $\textit{1}$ of these quaternions) 
\begin{equation}\label{xi isom1}
\langle\xi(X),\xi(Y)\rangle=\Re e\ \langle X\cdot\varphi^-,Y\cdot\varphi^-\rangle.
\end{equation}
This last identity easily gives
\begin{equation}\label{xi isom2}
\langle\xi(X),\xi(Y)\rangle=0\hspace{1cm}\mbox{and}\hspace{1cm}|\xi(Z)|^2=|Z|^2
\end{equation}
for all $X\in TM,$ $Y\in E$ and $Z\in E\cup TM.$ Thus $F=\int\xi$ is an isometry, and $\xi$ maps isometrically the bundle $E$ into the normal bundle of $F(M)$ in $\R^{4}.$ 

We now prove that $\xi$ preserves the normal connection and the second fundamental form: let $X\in TM$ and $Y\in\Gamma(E)\cup\Gamma(TM);$ then $\xi(Y)$ is a vector field normal or tangent to $F(M).$ Considering $\xi(Y)$ as a map $M\rightarrow\R^{4}\simeq \HH,$ we have
\begin{eqnarray}
d(\xi(Y))(X)&=&d\langle\langle Y\cdot\varphi^-,\varphi^+\rangle\rangle(X)\label{formula dxi}\\
&=&\langle\langle \nabla_XY\cdot\varphi^-,\varphi^+\rangle\rangle+\langle\langle Y\cdot\nabla_X\varphi^-,\varphi^+\rangle\rangle+\langle\langle Y\cdot\varphi^-,\nabla_X\varphi^+\rangle\rangle\nonumber
\end{eqnarray} 
where the connection $\nabla_XY$ denotes the connection on $E$ (if $Y\in\Gamma(E)$) or the Levi-Civita connection on $TM$ (if $Y\in\Gamma(TM)$). We will need the following formulas:
\begin{lem}
We have
\begin{eqnarray*}
\langle\ \langle\langle\nabla_XY\cdot\varphi^-,\varphi^+\rangle\rangle\ ,\ \xi(\nu)\ \rangle&=&\Re e\ \langle\nabla_XY\cdot\varphi^-,\nu\cdot\varphi^-\rangle\\&=&\Re e\ \langle\nabla_XY\cdot\varphi^+,\nu\cdot\varphi^+\rangle,
\end{eqnarray*}
$$\langle\ \langle\langle Y\cdot\nabla_X\varphi^-,\varphi^+\rangle\rangle\ ,\ \xi(\nu)\ \rangle=\Re e\ \langle Y\cdot\nabla_X\varphi^-,\nu\cdot\varphi^-\rangle$$
and
$$\langle\ \langle\langle Y\cdot\varphi^-,\nabla_X\varphi^+\rangle\rangle\ ,\ \xi(\nu)\ \rangle=\Re e\ \langle Y\cdot\nabla_X\varphi^+,\nu\cdot\varphi^+\rangle.$$
In the expressions above, $\langle.,.\rangle$ defined on $\HH$ for the left-hand side and $\Re e\ \langle.,.\rangle$ defined on $\Sigma$ for the right-hand side of each identity,  stand for the natural real scalar products.
\end{lem}
\begin{proof}
The first identity is a consequence of (\ref{xi isom1}) and the second identity may be obtained by a very similar computation observing that,   
 by (\ref{pty pairing1})-(\ref{pty pairing2}),
 $$\langle\ \langle\langle\nabla_XY\cdot\varphi^-,\varphi^+\rangle\rangle\ ,\ \xi(\nu)\ \rangle=\langle\ \langle\langle\nabla_XY\cdot\varphi^+,\varphi^-\rangle\rangle\ ,\ \langle\langle\nu\cdot\varphi^+,\varphi^-\rangle\rangle\ \rangle.$$ 
The last two identities may be obtained by very similar computations.
\end{proof}
\noindent From (\ref{formula dxi}) and the lemma, we readily get the formula 
\begin{equation}\label{formula dxi 2}
\langle d(\xi(Y))(X),\xi(\nu)\rangle=\frac{1}{2}\Re e\ \langle\nabla_XY\cdot\varphi,\nu\cdot\varphi\rangle+\Re e\ \langle Y\cdot\nabla_X\varphi,\nu\cdot\varphi\rangle.
\end{equation}
We first suppose that $X,Y\in\Gamma(TM).$ The first term in the right-hand side of the equation above vanishes in that case since $\nabla_XY\in\Gamma(TM),$ $\nu\in\Gamma(E)$. Recalling (\ref{B phi c=0}), we get then
\begin{eqnarray*}
\langle d(\xi(Y))(X),\xi(\nu)\rangle&=&\Re e\ \langle Y\cdot\nabla_X\varphi,\nu\cdot\varphi\rangle=\langle B(X,Y),\nu\rangle=\langle \xi(B(X,Y)),\xi(\nu)\rangle.
\end{eqnarray*}
Hence the component of $d(\xi(Y))(X)$ normal to $F(M)$ is given by
\begin{equation}\label{xi preserves B}
(d(\xi(Y))(X))^N=\xi(B(X,Y)).
\end{equation}
We now suppose that $X\in\Gamma(TM)$ and $Y\in\Gamma(E).$ We first observe that the second term in the right-hand side of equation (\ref{formula dxi 2}) vanishes. Indeed, if $(e_3,e_4)$ stands for an orthonormal basis of $E,$ for all  $i,j\in\{3,4\}$ we have
\begin{eqnarray*}
\Re e\ \langle e_i\cdot\nabla_X\varphi,e_j\cdot\varphi\rangle=-\Re e\ \langle \nabla_X\varphi,e_i\cdot e_j\cdot\varphi\rangle=-\Re e\ \langle \eta(X)\cdot\varphi,e_i\cdot e_j\cdot\varphi\rangle,
\end{eqnarray*}
which is a sum of terms of the form $\Re e\ \langle e\cdot\varphi,e'\cdot\varphi\rangle$ with $e$ and $e'$ belonging to $TM$ and $E$ respectively; these terms are therefore  all equal to zero. Thus, (\ref{formula dxi 2}) simplifies to
$$\langle d(\xi(Y))(X),\xi(\nu)\rangle=\frac{1}{2}\Re e\ \langle\nabla_XY\cdot\varphi,\nu\cdot\varphi\rangle=\langle\xi(\nabla_XY),\xi(\nu)\rangle.$$
Hence
\begin{equation}\label{xi preserves nabla}
(d(\xi(Y))(X))^N=\xi(\nabla_XY).
\end{equation}
Equations (\ref{xi preserves B}) and (\ref{xi preserves nabla}) mean that $\Phi_E=\xi$ preserves the second fundamental form and the normal connection respectively.
\end{proof}
\begin{rem}
The immersion $F:M\rightarrow\R^4$ given by the fundamental theorem is thus
$$F=\int\xi=\left(\int\ \xi_1,\int\ \xi_2,\int\ \xi_3,\int\ \xi_4\right).$$
This formula generalizes the classical Weierstrass representation: let $\alpha_1,\alpha_2,\alpha_3,\alpha_4$ be the $\C-$linear forms defined by
$$\alpha_k(X)=\xi_k(X)-i\xi_k(JX),$$
for $k=1,2,3,4,$ where $J$ is the natural complex structure of $M.$ Let $z$ be a conformal parameter of $M,$ and let $\psi_1,\psi_2,\psi_3,\psi_4:M\rightarrow\C$ be such that
$$\alpha_1=\psi_1dz,\ \alpha_2=\psi_2dz,\ \alpha_3=\psi_3 dz,\ \alpha_4=\psi_4dz.$$
By an easy computation using $D\varphi=\vec H\cdot\varphi,$ we see that $\alpha_1, \alpha_2, \alpha_3$ and $\alpha_4$ are holomorphic forms if and only if $M$ is a minimal surface ($\vec H=\vec 0$). Then
if $M$ is minimal, 
\begin{eqnarray*}
F&=&Re\left(\int\alpha_1,\int\alpha_2,\int\alpha_3,\int\alpha_4\right)\\
&=&Re\left(\int\psi_1dz,\int\psi_2dz,\int\psi_3dz,\int\psi_4dz\right)
\end{eqnarray*}
where $\psi_1,\psi_2,\psi_3,\psi_4$ are holomorphic functions. This is the Weierstrass representation of minimal surfaces.
\end{rem}
\begin{rem}
Theorem \ref{theorem second integration} also gives a spinorial proof of the fundamental theorem. We may integrate the Gauss, Ricci and Codazzi equations in two steps:
\\
\\1- first solving
\begin{equation}\label{nabla phi eta rem}
\nabla_X\varphi=\eta(X)\cdot\varphi,
\end{equation}
where 
$$\eta(X)=-\frac{1}{2}\sum_{j=1,2}e_j\cdot B(e_j,X)$$
(there is a unique solution in $\Gamma(\Sigma),$ up to the natural right-action of $Spin(4)$);
\\
\\2- then solving
$$dF=\xi$$
where $\xi(X)=\langle\langle X\cdot\varphi^-,\varphi^+\rangle\rangle$ (the solution is unique, up to translations). 
\\
\\Indeed, equation (\ref{nabla phi eta rem}) is solvable, since its conditions of integrability are exactly the Gauss, Ricci and Codazzi equations; see the proof of Theorem \ref{thm1}. Moreover, the multiplication of $\varphi$ on the right by a constant belonging to $Spin(4)$ in the first step, and the addition to $F$ of a constant belonging to $\R^4$ in the second step, correspond to a rigid motion in $\R^4.$ 
\end{rem}
\section{Surfaces in $\R^3$ and $S^3$.}
The aim of this section is to obtain as particular cases the spinor characterizations of T. Friedrich \cite{Fr} and B. Morel \cite{Mo} of surfaces in $\R^3$ and $S^3.$ Assume that  $M^2\subset \mathcal{H}^3\subset\R^4,$ where $\mathcal{H}^3$ is a hyperplane, or a sphere of $\R^4.$ Let $N$ be a unit vector field such that
 $$T\mathcal{H}=TM\oplus_{_\perp}\R N.$$ 
The intrinsic spinors of $M$ identify with the spinors of $\mathcal{H}$ restricted to $M,$ which in turn identify with the positive spinors of $\R^4$ restricted to $M:$ 
\begin{prop}
There is an identification
\begin{eqnarray*}
\Sigma M &\stackrel{\sim}{\rightarrow}& \Sigma^+_{|M}\\
\psi&\mapsto&\psi^*
\end{eqnarray*}
such that
$$(\nabla\psi)^*=\nabla(\psi^*)$$
and such that the Clifford actions are linked by
$$(X\cdot_{_M}\psi)^*=N\cdot X\cdot \psi^*$$
for all $X\in TM$ and all $\psi\in \Sigma M.$ 
\end{prop}
Using this identification, the intrinsic Dirac operator on $M$ defined by
$$D_M\psi:=e_1\cdot_{_M}\nabla_{e_1}\psi+e_2\cdot_{_M}\nabla_{e_2}\psi$$
is linked to $D$ by
$$(D_M\psi)^*=N\cdot D\psi^*.$$
If $\varphi\in\Gamma(\Sigma)$ is a solution of
$$D\varphi=\vec H\cdot \varphi\hspace{1cm}\mbox{ and }\hspace{1cm}|\varphi^+|=|\varphi^-|=1$$
then $\varphi^+\in\Sigma^+$ may be considered as belonging to $\Sigma M;$ it satisfies 
\begin{equation}\label{eqnphi+}
D_M\varphi^+ =N\cdot D\varphi^+=N\cdot\vec H\cdot\varphi^+.
\end{equation}
We examine separately the case of a surface in a hyperplane, and in a 3-dimensional sphere:
\\
\\1. If $\mathcal{H}$ is a hyperplane, then $\vec H$ is of the form $HN,$ and (\ref{eqnphi+}) reads
\begin{equation}\label{eqn Friedrich}
D_M\varphi^+ =-H\varphi^+.
\end{equation}
This is the equation considered by T. Friedrich in \cite{Fr}.
\\
\\2. If $\mathcal{H}=S^3,$ then $\vec H$ is of the form $HN-\nu,$ where $\nu$ is the outer unit normal of $S^3,$ and  (\ref{eqnphi+}) reads
\begin{equation}\label{eqn Morel}
D_M\varphi^+ =-H\varphi^+-i\overline{\varphi^+}.
\end{equation}
This equation is obtained by B. Morel in \cite{Mo}.
\\

Conversely, we now suppose that $\psi$ is an intrinsic spinor field on $M$ solution of (\ref{eqn Friedrich}) or (\ref{eqn Morel}). The aim is to construct a spinor field $\varphi$ in dimension 4 which induces an immersion in a hyperplane, or in a 3-sphere. Define $E=M\times\R^2,$ with its natural metric $\langle.,.\rangle$ and its trivial connection $\nabla'$, and consider $\nu,N\in \Gamma(E)$ such that
$$|\nu|=|N|=1,\hspace{1cm} \langle\nu,N\rangle=0\hspace{1cm}\mbox{and}\hspace{1cm}\nabla'\nu=\nabla'N=0.$$
We first consider the case of an hyperplane:
\begin{prop}
Let $\psi\in\Gamma(\Sigma M)$ be a solution of
$$D_M\psi =-H\psi$$
of constant length $|\psi|=1.$ There exists $\varphi\in\Gamma(\Sigma)$ solution of 
 \begin{equation}
 D\varphi=\vec H\cdot\varphi\hspace{1cm}\mbox{and}\hspace{1cm}|\varphi^+|=|\varphi^-|=1,
 \end{equation}
with $\vec H=HN,$ such that 
$$\varphi^+=\psi$$
and the normal vector field
$$\xi(\nu)=\langle\langle\nu\cdot\varphi^-,\varphi^+\rangle\rangle$$
has a fixed direction in $\HH.$ In particular, the immersion given by $\varphi$ belongs to the hyperplane $\xi(\nu)^{\perp}$ of $\HH.$ The spinor field $\varphi$ is  unique, up to the natural right-action of $S^3$ on $\varphi^-.$
\end{prop}
\noindent \textit{Proof:} define $\varphi=(\varphi^+,\varphi^-)$ by
$$\varphi^+=\psi,\hspace{1cm} \varphi^-=-\nu\cdot\psi.$$
We compute:
$$D\varphi^-=\nu\cdot D\varphi^+=\nu\cdot\vec H\cdot\varphi^+=\vec H\cdot\varphi^-,$$
$$\xi(\nu)=\langle\langle\nu\cdot\varphi^-,\varphi^+\rangle\rangle=1,$$
and, for all $X\in TM,$
\begin{equation*}
\xi(X)=\langle\langle X\cdot\varphi^-,\varphi^+\rangle\rangle=-\langle\langle X\cdot\nu\cdot\psi,\psi\rangle\rangle=\langle\langle \psi,X\cdot\nu\cdot\psi\rangle\rangle=\overline{\langle\langle X\cdot\nu\cdot\psi,\psi\rangle\rangle}=-\overline{\xi(X)},
\end{equation*}
that is $\xi(X)\in \Im\thmrm{m} (\HH),$ the hyperplane of pure imaginary quaternions. Thus $F=\int\xi$ also belongs to the hyperplane $\Im\thmrm{m} (\HH).$ Uniqueness is straightforward.
$\finpreuve$

We now consider the case of the 3-sphere:
\begin{prop}
Let $\psi\in\Gamma(\Sigma M)$ be a solution of
\begin{eqnarray*}
D_M\psi &=&-H\psi-i\overline{\psi}
\end{eqnarray*}
of constant length $|\psi|=1.$ There exists $\varphi\in\Gamma(\Sigma)$ solution of 
\begin{equation}
D\varphi=\vec H\cdot\varphi\hspace{1cm}\mbox{and}\hspace{1cm}|\varphi^+|=|\varphi^-|=1,
\end{equation}
with $\vec H=HN-\nu,$ such that 
$$\varphi^+=\psi$$
and the immersion $F$ defined by $\varphi$ is given by the unit normal vector field $\xi(\nu):$
$$F=\xi(\nu)=\langle\langle\nu\cdot\varphi^-,\varphi^+\rangle\rangle.$$
In particular $F(M)$ belongs to the sphere $S^3\subset\HH.$ The spinor field $\varphi$ is  unique, up to the natural right-action of $S^3$ on $\varphi^-.$
\end{prop}
\noindent \textit{Proof:} The system
$$\left\{\begin{array}{l}
F=\langle\langle\nu\cdot\varphi^-,\varphi^+\rangle\rangle\\
dF(X)=\langle\langle X\cdot\varphi^-,\varphi^+\rangle\rangle
\end{array}\right.$$
is equivalent to
$$\varphi^-=-\nu\cdot\varphi^+\cdot ,F$$
where $F:M\rightarrow\HH$ solves the equation
\begin{equation}\label{eqnF}
dF(X)=\beta(X)F
\end{equation}
in $\HH$, with
$$\beta(X)=-\langle\langle X\cdot\nu\cdot\varphi^+,\varphi^+\rangle\rangle.$$
By a direct computation, the compatibility equation  
$$d\beta(X,Y)=\beta(X)\beta(Y)-\beta(Y)\beta(X)$$
of (\ref{eqnF}) is satisfied, and equation (\ref{eqnF}) is solvable. Uniqueness is straightforward.
$\finpreuve$
\begin{rem}
Let $M$ be a minimal surface in $S^3$ and $N$ be such that
$$TM\oplus_{_\perp}\R N=TS^3.$$
For any $x\in S^3,$ denote by $\vec x=\stackrel{\rightarrow}{0x}$ the position vector of $x.$ At $x\in M,$ $\vec H=-\vec x.$ Thus, $M\subset S^3$ is represented by a solution $\varphi\in\Gamma(\Sigma)$ of
$$D\varphi=-\vec x\cdot\varphi.$$
The spinor field
$$\tilde\varphi:=(\varphi^+, N\cdot\varphi^+)$$ 
defines a surface of constant mean curvature $H=-1$ in $\Im\thmrm{m} (\HH)\simeq\R^3$. This is a classical transformation, described by B. Lawson in \cite{Lw}, and by T. Friedrich using spinors in dimension 3 in \cite{Fr}.
\end{rem}


\begin{thebibliography}{}
\bibitem{Ba} C. B\"ar, \emph{Extrinsic bounds for the eigenvalues of the Dirac operator}, Ann. Glob. Anal. Geom. \textbf{16} (1998) 573-596.
\bibitem{BFGK} H. Baum, T. Friedrich, R. Grunewald, I. Kath, \emph{Twistors and Killing spinors on Riemannian manifolds}, Teubner Texte zur Mathematik, Band 124 (1991). 
\bibitem{Fr0} T. Friedrich, \emph{On surfaces in four-space}, Ann. Glob. Analysis and Geom. \textbf{2} (1984) 257-287.
\bibitem{Fr} T. Friedrich, \emph{On the spinor representation of surfaces in Euclidean $3$-space}, J. Geom. Phys. \textbf{28} (1998) 143-157.
\bibitem{HZ} O. Hijazi and X. Zhang, \emph{Lower bounds for the eigenvalues of the Dirac operator, Part II. The submanifold Dirac operator}, Ann. Glob. Anal. Geom. \textbf{20} (2001) 163-181.
\bibitem{Ko} B.G. Konopelchenko, \emph{Weierstrass representations for surfaces in 4D spaces and their integrable deformations via DS hierarchy}, Ann. Glob. Anal. Geom. \textbf{18} (2000) 61-74. 
\bibitem{KS} R. Kusner and N. Schmidt, \emph{The spinor representation of surfaces in space}, Preprint arxiv dg-ga/9610005 (1996).
\bibitem{La} M.A. Lawn, \emph{A spinorial representation for Lorentzian surfaces in $\R^{2,1}$}, J. Geom. Phys. \textbf{58} (2008) no. 6, 683-700.
\bibitem{LR} M.A. Lawn and J. Roth, \emph{Spinorial characterization of surfaces in pseudo-Riemannian space forms}, Math. Phys. Anal. and Geom. \textbf{14} (2011) no. 3, 185-195.
\bibitem{Lw} H.B. Lawson, \emph{The global behaviour of minimal surfaces in $S^n$}, Ann. Math. \textbf{92} (1970) no. 2, 224-237.
\bibitem{Mo} B. Morel, \emph{Surfaces in $\mathbb{S}^3$ and $\mathbb{H}^3$ via spinors}, Actes du s\'eminaire de th\'eorie spectrale, Institut Fourier, Grenoble, \textbf{23} (2005) 9-22.
\bibitem{Ro} J. Roth, \emph{Spinorial characterizations of surfaces into 3-homogeneous manifolds}, J. Geom. Phys. \textbf{60} (2010) 1045-1061.
\bibitem{Ta} I. Taimanov, \emph{Surfaces of revolution in terms of solitons}, Ann. Glob. Anal. Geom. \textbf{15} (1997) 410-435.
\bibitem{Ta2} I. Taimanov, \emph{Surfaces in the four-space and the Davey-Stewartson equations}, J. Geom Phys. \textbf{56} (2006) 1235-1256. 
\bibitem{Te} K. Tenenblat, \emph{On isometric immersions of Riemannian manifolds}, Bol. Soc. Brasil. Mat. \textbf{2} (1971) no. 2, 23-36.
\end{thebibliography}
\end{document}